\documentclass[11pt,a4paper]{article}
\usepackage[english]{babel}
\usepackage[utf8]{inputenc}
\usepackage[T1]{fontenc}
\usepackage{amsmath}
\usepackage{amsthm}
\usepackage{amssymb, color}
\usepackage[a4paper,top=1.5in,bottom=1.5in,left=1.in,right=1.in,marginparwidth=1.75cm]{geometry}
\usepackage{graphicx}
\usepackage{enumerate}
\usepackage[section]{placeins}
\usepackage{dsfont}
\usepackage{mathrsfs}
\usepackage{varioref}
\usepackage{hyperref}
\usepackage[capitalise]{cleveref}
\usepackage{hyperref}
\usepackage{todonotes}
\usepackage{array}
\usepackage{makecell}
\usepackage{mathtools}
\usepackage{etoolbox}

\newbool{thesis}
\boolfalse{thesis}

\theoremstyle{plain}
\newtheorem{theorem}{Theorem}[section]
\newtheorem{proposition}[theorem]{Proposition}

\newtheorem{corollary}[theorem]{Corollary}

\theoremstyle{definition}
\newtheorem{definition}[theorem]{Definition}
\newtheorem{example}[theorem]{Example}
\newtheorem{remark}[theorem]{Remark}

\newcommand{\R}{\mathbb{R}}

\newcommand{\N}{\mathbb{N}}

\newcommand{\cM}{\mathcal{M}}

\newcommand{\cC}{\mathcal{C}}

\newcommand{\sB}{\mathscr{B}}

\newcommand{\da}{\delta}

\newcommand{\e}{\varepsilon}

\numberwithin{equation}{section}

\newcommand{\la}{\lambda}

\newcommand{\de}{\textnormal{ d}}
\newcommand{\den}{\textnormal{d}}

\newcommand{\I}{\mathds{1}}

\newcommand{\curvebrac}[1]{\left( #1 \right)}

\newcommand{\modd}[1]{\left| #1 \right|}
\newcommand{\norm}[1]{\left\lVert#1\right\rVert}

\newcommand{\BV}{\textnormal{BV}}
\newcommand{\BVloc}{\textnormal{BV}_{\textnormal{loc}}}

\DeclareMathOperator*{\vague}{v-lim}

\definecolor{vargreen}{RGB}{0,150,0}

\begin{document}

\title{A continuity theorem for generalised signed measures with an application to Karamata's Tauberian theorem}
	
	\author{Martin Herdegen, Gechun Liang, Osian Shelley\thanks{All authors: University of Warwick, Department of Statistics, Coventry, CV4 7AL, UK; \{m.herdegen, g.liang, o.d.shelley\}@warwick.ac.uk}}
	\date{\today}
	
	\maketitle

\begin{abstract}
The Laplace transforms of positive measures on $\R_{+}$ converge if and only if their distribution functions converge at continuity points of the limiting measure. We extend this classical continuity theorem to the case of generalised signed Radon measures. The result for the signed case requires some additional conditions, which follow from recent results on vague convergence of signed Radon measures. As an application, we introduce a novel Tauberian condition for generalised signed Radon measures that extends Karamata's Tauberian theorem.\\
    
\end{abstract}


\bigskip
\noindent\textbf{Keywords:} Continuity theorem,\ Tauberian theorem,\ signed measures,\ vague convergence.

\section{Introduction}\label{section:introduction}
It is well known that Laplace transforms of positive measures on $\R_{+} := [0, \infty)$ converge if and only if their distribution functions converge at continuity points of the limiting measure; see e.g.~\cite[XIII.1, Theorem 2a]{feller_introduction_1971}. In this paper we extend this so-called continuity theorem to the case of generalised signed Radon measures. Our proof relies on recent results linking the vague convergence of signed measures to the convergence of their distribution functions; see \cite{herdegen_vague_2022}. The latter results explain the need for some extra conditions in the continuity theorem for the signed case (that are automatically satisfied in the positive case). We illustrate the use of this new continuity theorem to deduce a novel Tauberian condition that allows for an extension of Karamata's Tauberian theorem; see e.g. \cite[XIII.5, Theorem 1]{feller_introduction_1971}.
\medskip{}

The remainder of the paper is organised as follows.
In Section \ref{subsection:notation_and_defnitinions}, we introduce some key pieces of notation. In
\Cref{section:continuity}, we state and prove the \textit{continuity theorem} (\cref{thm:Continuity_Thm_Signed}) for generalised signed Radon measures. In \cref{section:Proof}, we state and prove an extension of Karmata's Tauberian theorem (\cref{thm:Karamata_Tauberian}). The \cref{section:appendix} contains some auxiliary results.

\subsection{Preliminary definitions and notation}\label{subsection:notation_and_defnitinions}
Recall that a positive measure $\mu$ on $({\R_{+}}, \sB({{\R_{+}}}))$ is called a \emph{Radon measure} if it is \emph{locally finite} and \emph{inner regular}, i.e., 
  \begin{enumerate}[(i)]
      \item for any $x \in {\R_{+}}$, there exists an open neighbourhood $U$ of $x$ (in $\R_+)$ such that $\mu(U) < \infty$;
      \item for each $A \in \sB({\R_{+}})$
      \begin{equation*}
        {\mu}(A) = \sup\{{\mu}(K): K\in \sB({\R_{+}}), ~K \textnormal{ compact},~ K \subset A\}.
      \end{equation*}
  \end{enumerate}

\begin{definition}\label{def:local_signed} 
A mapping $\mu:\sB({\R_{+}}) \to \R \cup \{\infty\}$ is called a \emph{generalised signed Radon measure} if 
there exist mutually singular Radon measures $\mu^+$ and $\mu^-$ such that for any $A\in \sB({\R_{+}})$,
    \begin{equation*}
        \mu(A) := \begin{cases}
            \mu^+(A)  - \mu^-(A) &\textnormal{if both } \mu^+(A),\mu^-(A)  < \infty,\\
            +\infty & \textnormal{otherwise. }
        \end{cases}
    \end{equation*}
We set $|\mu| := \mu^+ + \mu^-$ and $\norm{\mu} := |\mu|(\R_+)$. We say that $\mu$ is a \emph{finite signed Radon measure} if $\norm{\mu} < \infty$. We denote the space of generalised signed Radon measures on $({\R_{+}}, \sB({{\R_{+}}}))$ by $\cM^{*}$ and the set of all finite signed Radon measures on $({\R_{+}}, \sB({{\R_{+}}}))$ by $\cM$.
    \end{definition}

\begin{remark} 
(a) A generalised signed Radon measure need not be a signed measure in the classical sense as we allow both $\mu^+$ and $\mu^-$ to have infinite mass. By contrast, a signed Radon measure is a finite signed measure in the classical sense, and $|\mu|$ is the total variation measure.

(b) It is easy to check that a mapping $\mu:\sB({\R_{+}}) \to \R \cup \{\infty\}$ is in $\cM^{*}$ if and only if the restrictions $\mu|_{[0, T]}:\sB([0, T]) \to \R \cup \{\infty\}$ are in $\cM$ for all $T \geq 0$.
\end{remark}

For an interval $I \subset \R$, we denote by $\BV(I)$ the space of all functions $f: I \to \R$ that are of of bounded variation on $I$. Moreover, we denote by $\BVloc(I)$, the space of all  $f: I \to \R$ that are locally of bounded variation, i.e., $f|_{[a, b]} \in \BV([a,b])$ for any $[a, b] \subset I$. 

\medskip{}

For $\mu \in \cM^*$, we define its \emph{distribution function} $F_{\mu} \in \BVloc({\R_{+}})$ by 
\begin{equation*}
    F_{\mu}(x) := \begin{cases}
        0\hspace*{1.55cm} \text{for } x = 0,\\
        \mu([0,x]) \quad \text{for }x>0.
    \end{cases}
\end{equation*}
Note that $F_{\mu}$ is right continuous and satisfies
\begin{equation*}
    F_{\mu}(b) - F_{\mu}(a) = \mu((a,b]), \quad \text{for all } a < b \in \R.
\end{equation*}

Throughout the remainder of this paper, we will be considering measures in $\cM^*$ whose Laplace transform is well defined on all of $(0,\infty)$. Thus, we set
\begin{equation*}
    \cM_\Psi :=\left\{ \mu \in \cM^* : \int_{\R_{+}} e^{-\la x}\modd{\mu}(\den x) < \infty \text{ for all } \la > 0\right\}.
\end{equation*}
For $\mu \in \cM_\Psi$, we define its \emph{Laplace transform} $\Psi_\mu \in \BVloc((0,\infty))$ by

\begin{equation*}
\Psi_\mu(\la) := \int_{\R_{+}}e^{-\la x}\mu(\den x), \quad \lambda > 0.
\end{equation*}


\section{Continuity Theorem}\label{section:continuity}
In this section, we state and prove the continuity theorem for generalised signed Radon measures. This extends the classical continuity theorem from Feller \cite[XIII.1, Theorem 2a]{feller_introduction_1971}.

A key tool tool for our proof is the concept of vague convergence. To this end, we need to introduce some further pieces of notation: Let $C({\R_{+}})$ be the space of all continuous $\R$-valued functions on ${\R_{+}}$, $C_0({\R_{+}})$ the subspace of all $f \in C({\R_{+}})$ such that for any $\e>0$, there exists a compact set $K_\e\in \sB({\R_{+}})$ with $\modd{f} < \e$ on $K_\e^c$, and $C_c({\R_{+}})$ the subspace of all $f \in C({\R_{+}})$ such that  $f$ has compact support. Clearly, we have the inclusions $C_c({\R_{+}}) \subseteq C_0({\R_{+}})\subseteq C({\R_{+}})$.

\begin{definition}\label{def:vague}
    For any $\mu \in \cM$, define the map $I_\mu:C_b({\R_{+}}) \to \R$ by
\begin{equation*}
   I_{\mu}(f)  =  \int_{\R_{+}} f \de \mu.
\end{equation*}
Then, we say that a sequence $\{\mu_n\} \subset \cM$ converges \emph{vaguely} to $\mu$, if $I_{\mu_n}(f) \to I_{\mu}(f)$ for all $f \in C_c({\R_{+}})$, which we denote by $$\vague_{n \to \infty}\mu_n = \mu;$$
\end{definition}

\begin{remark}
Note that sometimes $C_0({\R_{+}})$ rather than $C_c({\R_{+}})$ is used as the space of test functions for vague convergence in the literature. However, if $\sup_{n \in \N}\norm{\mu_n} < \infty$, the two definitions are equivalent; see {\cite[Proposition 1.3 and Remark 1.4]{herdegen_vague_2022}} for more details.
\end{remark}

The following definition is based on {\cite[Definition 3.6]{herdegen_vague_2022}}. It provides a sufficient condition to ensure that distribution functions of signed measures converge at continuity points; see \cite[Theorem 3.9]{herdegen_vague_2022} for details.
\begin{definition}\label{def:no_mass}
    We say that $\{\mu_n\} \subset \cM^{*}$ is \emph{right-equicontinuous} at a point $x \in \R_{+}$, if for any $\e > 0$ there exists $h > 0$ such that for all $\delta < h$
   \begin{equation*}
       \limsup_{n \to \infty}| {\mu_n}((x,x+\delta]) | \leq \e.
   \end{equation*}
\end{definition}

We can now state and prove the continuity theorem for signed measures.
\begin{theorem}[Continuity Theorem]\label{thm:Continuity_Thm_Signed}
   Let $\{\mu_n\} \cup \{\mu\} \subset \cM_{\Psi}$ be such for any $\la > 0$,
   \begin{align}
       \limsup_{n \to \infty} \Psi_{\modd{\mu_n}}(\la) &< \infty.\label{eq:Continuity_Thm_Signed_bounded_condition}
   \end{align}
   \begin{enumerate}[\normalfont(a)]
       \item Suppose $\Psi_{\mu_n}(\la) \to \Psi_\mu(\la)$ for all $\la > 0$. If $\{\mu_n\}$ is right-equicontinuous  at a continuity point $x \in \R_{+}$ of $\mu$,  then $F_{\mu_n}(x) \to F_\mu(x)$.

       \item If $F_{\mu_n} \to F_\mu$ a.e., then $\Psi_{\mu_n} \to \Psi_{\mu}$.
   \end{enumerate}
\end{theorem}

\begin{proof}
For $\nu \in \cM^*$ and $\e > 0$, define $\nu^{(\e)} \in \cM$ by $\nu^{(\e)}(\den x) := {e^{-\e x}} \nu(\den x)$. Then for each $\e > 0$, \eqref{eq:Continuity_Thm_Signed_bounded_condition} and the fact that $\Psi_{|\mu_n|}(\e) < \infty$ for each $n \in \N$ give
    \begin{equation}
\label{eq:pf:thm:Continuity_Thm_Signed:bound}
        \sup_{n \in \N}\norm{\mu_n^{(\e)}} =   \sup_{n\in \N}\Psi_{\modd{\mu_{n}}}(\e)<\infty.
    \end{equation}
Moreover, if in addition $\vague_{n \to \infty}\mu_n^{(\e)} = \nu$ for some $\nu \in \cM$, then for each $\lambda > 0$, \eqref{eq:pf:thm:Continuity_Thm_Signed:bound}, \cite[Proposition 1.3]{herdegen_vague_2022}  and the fact that $\exp(-\lambda \cdot) \in C_0(\R_{+})$ give
    \begin{equation}
\label{eq:pf:thm:Continuity_Thm_Signed:Psi cond}
\lim_{n \to \infty} \Psi_{\mu_n^{(\e)}}(\la) = \lim_{n \to \infty} \int_0^\infty e^{-\lambda t} \mu_n^{(\e)}(\mathrm{d} t) = \int_0^\infty e^{-\lambda t} \nu(\mathrm{d} t) = \Psi_{\nu}(\lambda)
    \end{equation}

    \normalfont(a) By a simple generalisation of \cite[Theorem 3.8]{herdegen_vague_2022}, it suffices to show that  
    \begin{equation*}
\vague_{n \to \infty}\mu_{n} =\mu.
\end{equation*}
To establish the latter, it suffices to show that for each $\e > 0$, 
    \begin{equation}
\label{eq:pf:thm:Continuity_Thm_Signed:vlim:02}
\vague_{n \to \infty}\mu^{(\e)}_{n} =\mu^{(\e)},
\end{equation}
Indeed, fix $\e > 0$ and $f \in C_c({\R_{+}})$. Then the vague convergence of $\{ \mu^{(\e)}_n \}$ and the fact that $f \exp(-\e \cdot) \in C_c(\R_{+})$ gives
    \begin{align*}
        \int_{{\R_{+}}}f \de \mu_n =\int_{{\R_{+}}} f(x)e^{\e x} \mu^{(\e)}_n(\den x)
        \rightarrow \int_{{\R_{+}}} f(x)e^{\e x} \mu^{(\e)}(\den x) = \int_{{\R_{+}}}f \de \mu.
    \end{align*}

To establish \eqref{eq:pf:thm:Continuity_Thm_Signed:vlim:02}, fix $\e> 0$. By the subsequence criterion, it suffices to show that for every subsequence $\{n_k\}$, there exists a further subsequence $\{n_{k_\ell}\}$ such that
    \begin{equation}
\label{eq:pf:thm:Continuity_Thm_Signed:vlim:03}
\vague_{\ell \to \infty}\mu^{(\e)}_{n_{k_\ell}} = \mu^{(\e)}.
\end{equation}
So let $\{n_k\}$ be a subsequence. Then \eqref{eq:Continuity_Thm_Signed_bounded_condition} gives $\sup_{k \in \N}\norm{\mu^{(\e)}_{n_k}} =   \sup_{k \in \N}\Psi_{\modd{\mu_{n_k}}}(\e)<\infty$. Thus, by {\cite[Theorem 1.2]{herdegen_vague_2022}} we can view $\{\mu^{(\e)}_{n_k}\}$ as a bounded family in $(C_0({\R_{+}}))^*$. Hence, the Banach-Alaoglu theorem implies that $\{\mu^{(\e)}_{n_k}\}$ is $\sigma((C_0({\R_{+}}))^*,C_0({\R_{+}}))$-compact. Furthermore, as a consequence of \Cref{thm:stoneWeierstrass}, $C_0({\R_{+}})$ is separable, whence the sequence is sequentially compact. Thus, there exists a subsequence $\{n_{k_\ell}\}$ and $\tilde \mu \in \cM$ such that 
\begin{equation}
\vague_{\ell \to \infty}\mu^{(\e)}_{n_{k_\ell}} =\tilde \mu \label{eq:pf:thm:Continuity_Thm_Signed:tilde mu}
\end{equation}

We proceed to show that $\tilde \mu = \mu^{(\e)}$. Since by \cref{thm:signedCharacterisation} each element in $\cM_{\Psi}$ is uniquely characterised by its Laplace transform, it suffices to show that $\Psi_{\tilde \mu}(\lambda) =\Psi_{\mu^{(\e)}}(\lambda)$ for each $\lambda > 0$. So fix $\lambda > 0$. Then the hypothesis together with  \eqref{eq:pf:thm:Continuity_Thm_Signed:Psi cond} (applied to the subsequence $\{n_{k_\ell}\}$) give
    \begin{equation*}
      \Psi_{\mu^{(\e)}}(\la) = \Psi_{\mu}(\la+\e)= \lim_{\ell \to \infty}  \Psi_{\mu_{n_{k_\ell}}}(\la + \e)  =
\lim_{\ell \to \infty}  \Psi_{\mu^{(\e)}_{n_{k_\ell}}}(\la)   = \Psi_{\tilde \mu}(\la).
    \end{equation*}

    {\normalfont(b)} It suffices to show that for each $\e > 0$,  
\begin{equation}
\label{eq:pf:thm:Continuity_Thm_Signed:F cond}
F_{\mu^{(\e)}_n} \to F_{\mu^{(\e)}} \;\text{a.e.}
\end{equation}
Indeed, \cite[Theorem 3.8 (a)]{herdegen_vague_2022} then gives $\textnormal{v-lim}_{n \to \infty}\mu^{(\e)}_n = \mu^{(\e)}$, and this in turn together with \eqref{eq:pf:thm:Continuity_Thm_Signed:Psi cond} yields for $\lambda > 0$,
    \begin{equation*}
       \lim_{n \to \infty} \Psi_{\mu_n}(\la + \e) =  \lim_{n \to \infty} \Psi_{\mu^{(\e)}_n}(\la)  = \Psi_{\mu^{(\e)}}(\la ) = \Psi_{\mu}(\la + \e).
    \end{equation*}
    Since $\e$ is arbitrary, the claim follows.

To establish \eqref{eq:pf:thm:Continuity_Thm_Signed:F cond}, fix $\e > 0$ and let $t > 0$ be such that $\lim_{n \to \infty} F_{\mu_n}(t) = F_\mu(t)$.
An integration by parts gives
    \begin{align*}
        F_{\mu^{(\e)}_n}(t) = \int_{[0, t]} e^{-\e x} \mu (\den x)
        = e^{-\e t}F_{\mu_n}(t) + \e\int_0^t e^{-\e x} F_{\mu_n}(x)\de x.
    \end{align*}
Using that
\begin{equation*}
\sup_{x \in [0, t]} \sup_{n \in \N} F_{\modd{\mu_n}}(x) \leq e^{t} \sup_{n \in \N}\Psi_{\modd{\mu_n}}(1) < \infty
\end{equation*}
by \eqref{eq:Continuity_Thm_Signed_bounded_condition} for $\lambda = 1$, the hypothesis, dominated convergence and an integration by parts give
    \begin{equation*}
    \lim_{n \to \infty}  F_{\mu^{(\e)}_n}(t) = e^{-\e t}F_{\mu}(t) + \e\int_0^t e^{-\e x} F_{\mu}(x)\de x = F_{\mu^{(\e)}}(t). \qedhere
    \end{equation*}

\end{proof}

\begin{remark}\label{remark:continuity_theorem}
    \begin{enumerate}[(i)]
        \item One can also prove part (a) by using Helly's selection theorem \cite[Theorem 2.35]{leoni_first_2017} to find a convergent subsequence of the family $\{F_{\mu_n^{(\e)}}\}$. The limiting function will be of bounded variation, and right continuous due to the right-equicontinuity condition. Thus, it will be the distribution function of a measure. The remainder of the proof is identical. 
        
        \item When the measures are positive, condition \eqref{eq:Continuity_Thm_Signed_bounded_condition} is clearly satisfied under the hypothesis of part (a). Moreover, the right-equicontinuity condition of part (a) is no longer needed due to \cite[Theorem 3.2]{herdegen_vague_2022}.
        
        \item Under the assumption $\Psi_{\mu_n}(\la) \to \Psi_\mu(\la)$ for all $\la > 0$, (\ref{eq:Continuity_Thm_Signed_bounded_condition}) is trivially satisfied when the measures are positive. Otherwise, a sufficient condition for \eqref{eq:Continuity_Thm_Signed_bounded_condition} is that there exists $\da \in [0,1)$ such that either $\Psi_{\mu_n^-}(\la) < \da \Psi_{\mu_n^+}(\la)$ for each $\lambda > 0$ or $\Psi_{\mu_n^+}(\la) < \da \Psi_{\mu_n^-}(\la)$ for each $\lambda > 0$. We only establish the first case. Then
        \begin{align*}
        \limsup_{n \to \infty}\Psi_{\modd{\mu_n}}(\la) &\leq \limsup_{n \to \infty}(1+\da)\Psi_{\mu_n^+}(\la) \leq \curvebrac{\frac{1+\da}{1-\da}}\limsup_{n \to \infty} \Psi_{\mu_n}(\la) \\
        &= \curvebrac{\frac{1+\da}{1-\da}} \Psi_\mu(\lambda) < \infty.
        \end{align*}
    \end{enumerate}
\end{remark}

The following example illustrates that the right-equicontinuity condition is indeed needed for Theorem \ref{thm:Continuity_Thm_Signed}(a).

\begin{example}
Let $\{\mu_n\} \cup \{\mu\} \in \cM_\Psi$ be defined by $\mu_n:= \delta_{x} - \delta_{x+\frac{1}{n}}$ for for some $x>0$, and $\mu\equiv 0$. Note that $\{\mu_n\}$ is not right-equicontinuous at $x$. Indeed, for any $\delta > 0$
\begin{equation*}
    |\mu_n((x,\delta])| = 1\quad \text{ for } \quad n \geq \frac{1}{\delta}.
\end{equation*}
It is straightforward to check that  $$\limsup_{n \to \infty} \Psi_{|\mu_n|}(\lambda)  = \limsup_{n \to \infty}  \left(e^{-\la x} +  e^{-\lambda\left(x+\frac{1}{n} \right)}\right) = 2 e^{-\la x} < \infty, \quad \lambda > 0$$ and
    \begin{equation*}
        \Psi_{\mu_n}(\lambda) = e^{-\la x} -  e^{-\lambda\left(x+\frac{1}{n}\right)} \to 0 = \Psi_\mu(\lambda), \quad \lambda > 0.
    \end{equation*}
    However, $x$ is a continuity point of $\mu$ and
    \begin{equation*}
        F_{\mu_n}(x) =  1  \not \to 0 = F_{\mu}(x). 
    \end{equation*}
\end{example}

\section{Application: Karamata's Tauberian Theorem}\label{section:Proof}

In the study of regular variation, Karamata's Tauberian Theorem for Laplace-Stieltjes transforms is a classical result; see \cite[Theorem 1]{feller_classical_1963}, \cite[Theorem 1.7.1]{bingham_regular_1989}, \cite[Theorem 1]{konig_neuer_1960}.  It relates regular variation of a positive monotone functions $F$ at infinity to the regular variation of its Laplace transforms $\Psi_{\mu_F}$ at zero. Due to the relationship between positive monotone functions and positive Radon measures, the theorem can be also stated for measures.

\begin{theorem}\label{thm:Karamata_Tauberian}
    Let $\mu \in \cM_\Psi$ be a positive measure and $\rho \geq 0$. The limit statements
   \begin{equation}\label{eq:signedPsiLimit}
            \lim_{\tau\downarrow 0} \frac{\Psi_\mu(\tau \la)}{\Psi_\mu(\tau)} = \frac{1}{\la^\rho}, \quad \lambda > 0,
            \end{equation}
and 
        \begin{equation}\label{eq:signedFLimit}
            \lim_{t \uparrow \infty} \frac{F_\mu(tx)}{F_\mu(t)}=x^\rho, \quad x > 0.
            \end{equation}
imply each other. In either case, we also have 
            \begin{equation}\label{eq:signedTheRelationship}
                \Psi_\mu(t^{-1}) \sim F_\mu(t)\Gamma(\rho + 1) \quad \text{as } t   \to \infty.
            \end{equation}
\end{theorem}

According to Feller, Theorem \ref{thm:Karamata_Tauberian} has a `glorious history' \cite[Section XIII.5]{feller_introduction_1971}, even though it is often omitted from modern books on probability theory. The two implications are usually separated, where \cref{eq:signedFLimit} $\Rightarrow$ \cref{eq:signedPsiLimit} is called an \emph{Abelian} theorem, whilst \cref{eq:signedPsiLimit} $\Rightarrow$ \cref{eq:signedFLimit} is called a \emph{Tauberian }theorem. Looking to its origins we see that the Tauberian implication caused more difficulty. It was first proved via lengthy calculations in 1914 by Hardy and Littlewood in their famous paper \cite{hardy_tauberian_1914}. Karamata simplified their proof in \cite{karamata_sur_1930}, and subsequently introduced the present day class of regularly varying functions.

\subsection{Karamata's Tauberian theorem for generalised signed measures}
Karamata's theorem can be extended to functions of local bounded variation, and hence generalised signed measures. While there are some results,  see e.g. \cite[Section 4.0, 5]{bingham_regular_1989}, they do not seem to be well known. They always require some additional conditions in the Tauberian direction, usually referred to as \textit{Tauberian conditions}. The latter are needed to account for the lack of monotonicity of $F_{\mu}$ in the proof of Theorem \ref{thm:Karamata_Tauberian}. 

We proceed to apply our continuity theorem to obtain a version of Karamata's Tauberian theorem for generalised signed measures.

\begin{theorem}[Karamata's Tauberian theorem for generalised signed measures]\label{thm:Karamata_Tauberian_Signed}
    Let $\mu \in \cM_{\Psi}$ and $\rho \geq 0$. 
    \begin{enumerate}[\normalfont(a)]
        \item Suppose that
        \begin{equation}\label{eq:condition_1} 
            \liminf_{\tau \downarrow 0}\frac{ \modd{\Psi_{\mu} (\tau )}  }{\Psi_{\modd{\mu}} (\tau ) } > 0
    \end{equation}
    as well as
            \begin{equation}\label{eq:extraCondtion}
                \limsup_{h \downarrow 0 }\limsup_{\tau \downarrow 0} \modd{ \frac{F_\mu(\tau^{-1}(x+h)) - F_\mu(\tau^{-1}x) }{ \Psi_{\mu}(\tau)}}=0, \quad x > 0.
            \end{equation}
    Then \eqref{eq:signedPsiLimit} implies \eqref{eq:signedFLimit}.
        \item \eqref{eq:signedFLimit} implies \eqref{eq:signedPsiLimit}.
    \end{enumerate}
Moreover, in either case, we have the asymptotic relationship \eqref{eq:signedTheRelationship}.
\end{theorem}

\begin{remark}\label{remark:Thm_Signed_Tauberian}
    \begin{enumerate}
        \item The limit statements \eqref{eq:signedPsiLimit} and \eqref{eq:signedFLimit}  contain the implicit assumption that $\Psi_\mu(\tau)$ and $F_\mu(t)$ are non-zero for sufficiently small $\tau$ and sufficiently large $t$, respectively. In particular \eqref{eq:signedPsiLimit} implies that $\Psi_\mu$ has one sign near the origin, and \eqref{eq:signedFLimit} implies that $F_\mu$ has one sign near infinity.
         
        \item The method of proof of Theorem \ref{thm:Karamata_Tauberian_Signed}(a) is due to König \cite{konig_neuer_1960} and Feller \cite{feller_classical_1963}. We extend the approach using Theorem \ref{thm:Continuity_Thm_Signed}. The argument for Theorem \ref{thm:Karamata_Tauberian_Signed}(b) is inspired by \cite[Theorem 1.7.6]{bingham_regular_1989}.
    \end{enumerate}
\end{remark}

\begin{proof}[Proofs of \cref{thm:Karamata_Tauberian_Signed}]
Throughout the proof, let $\{\tau_n\}\subset (0,\infty)$ be an arbitrary null sequence and define the sequence $\{t_n\} \subset (0, \infty)$ by $t_n := \tau_n^{-1}$. 

    {\normalfont(a)} 
    By \eqref{eq:condition_1}, we may assume without loss of generality that $\Psi_\mu(\tau_n) \neq 0$ for all $n \in \N$. Thus, for each $n \in \N$, we may define $\nu_n \in \cM_{\Psi}$ by
    \begin{align*}
    \nu_n(\den x):= \frac{\mu( \den (t_n x) )}{\Psi_\mu(\tau_n)}.
    \end{align*}
    Moreover, define the measure $\nu \in \cM_{\Psi}$ by
    \begin{equation*}
    \nu(\den x) := \begin{cases}
    \frac{x^{\rho-1}}{\Gamma(\rho)}\de x & \text{if  } \rho > 0,\\
    \delta_0(\den x) & \text{if  } \rho = 0.\\
    \end{cases}
    \end{equation*}
    Then \eqref{eq:signedPsiLimit} gives 
    \begin{equation*}
    \lim_{n \to \infty}  \Psi_{\nu_n}(\la) = \lim_{n \to \infty} \frac{\Psi_{\mu}(\tau_n\la)}{\Psi_\mu(\tau_n)} =  \la^{-\rho} = \Psi_\nu(\la).
    \end{equation*}
    This together with \eqref{eq:condition_1} in turn yields
    \begin{align*}
        \limsup_{n \to \infty}  \Psi_{\modd{\nu_n}}(\la)  &=  \limsup_{n \to \infty} \frac{\Psi_{\modd{\mu}}(\tau_n\la)}{ {\Psi_{{\mu}}(\tau_n) }} = \la^{-\rho}\limsup_{n \to \infty} \frac{\Psi_{\modd{\mu}}(\tau_n\la)}{ {\Psi_{{\mu}}(\tau_n\la) }}< \infty, \quad \lambda > 0.\label{eq:guarantee_bounded_Laplace}
    \end{align*}
    Moreover, $\{\nu_n\}$ is right-equicontinuous  at all points in  $(0,\infty)$. Indeed, let $x \in (0, \infty)$. Then \eqref{eq:extraCondtion} gives
    \begin{align}
        \limsup_{h \downarrow 0}\limsup_{n \to \infty}| {\nu_n}((x,x+h]) |
        &=  \limsup_{h \downarrow 0 }\limsup_{\tau \downarrow 0} \frac{\modd{ F_\mu(\tau^{-1}(x+h)) - F_\mu(\tau^{-1}x) }}{ \Psi_{\mu}(\tau)}= 0.\nonumber
    \end{align}
    It now follows from \cref{thm:Continuity_Thm_Signed}\normalfont{(a)} that $F_{\nu_n} \to F_\nu$ on $(0,\infty)$. Recalling the definition of $\nu_n$, this implies that
    \begin{equation}
    \label{eq:pf:thm:Karamata_Tauberian_Signed:key step a}
    \lim_{n \to \infty} \frac{F_{\mu}(t_n x)}{\Psi_\mu(\tau_n)} = \lim_{n \to \infty} F_{\nu_n}(x) =  F_{\nu}(x) = \frac{x^\rho}{\Gamma(\rho+1)}, \quad  x > 0.
    \end{equation}
    Finally, \eqref{eq:pf:thm:Karamata_Tauberian_Signed:key step a} for $x= 1$ gives \eqref{eq:signedTheRelationship}, and then combining  \eqref{eq:pf:thm:Karamata_Tauberian_Signed:key step a}  and \eqref{eq:signedTheRelationship} yields
    \eqref{eq:signedFLimit} via
    \begin{equation*}
    \lim_{n \to \infty}    \frac{F_{\mu}(t_n x)}{F_{\mu}(t_n)} = \lim_{n \to \infty}  \frac{F_{\mu}(t_n x)}{\Psi_{\mu}(\tau_n)}\frac{\Psi_{\mu}(\tau_n)}{F_{\mu}(t_n)} = \frac{x^\rho}{\Gamma(\rho+1)}\Gamma(\rho+1)= x^\rho, \quad x > 0.
    \end{equation*}

    {\normalfont(b)} By Remark \ref{remark:Thm_Signed_Tauberian}(i), without loss of generality we choose $X > 0$ such that $\I_{[X,\infty)}F_{\mu}$ is strictly positive. Define the positive measure $\xi \in \cM_\Psi$ via
    \begin{equation}\label{eq:pf:Continuity_Thm_Signed:F_limit} 
        F_{\xi}(x) := \int_0^x \I_{[X,\infty)}(t)F_{\mu}(t)\de t.
    \end{equation}
    By the Characterisation Theorem for regularly varying functions \cite[Theorem 1.4.1]{bingham_regular_1989}, (\ref{eq:signedFLimit}) implies there exists a slowly varying function $l$ such that $F_{\mu}(x) = x^\rho l(x)$. In particular, $\I_{[X,\infty)}(x)F_\mu(x) \sim x^\rho l(x)$, and so using \cite[Proposition 1.5.8]{bingham_regular_1989} it follows 
    \begin{equation*}
        F_{\xi}(x) \sim \frac{x^{\rho +1}l(x)}{(\rho + 1)}  \quad\text{ as }x \to \infty.   
    \end{equation*}
    
    Since $\xi$ is a positive measure, Theorem \ref{thm:Karamata_Tauberian} lets us infer that $\Psi_{\xi}(\tau ) \sim  F_{\xi}(\tau^{-1})\Gamma(\rho + 2)$ as $\tau \to 0$, whence 
    \begin{equation}\label{eq:pf:thm:Continuity_Thm_Signed:ratio b}
        \Psi_{\xi}(\tau) \sim \Gamma(\rho + 1) l(1/\tau) \tau^{-(\rho + 1)}  \quad\text{ as }\tau \to 0.     
    \end{equation}
    Noting that  $\Psi_{\xi}(\lambda) = \frac{1}{\lambda} \int_X^\infty e^{-\lambda x} \mu(\den x ) $, \eqref{eq:pf:thm:Continuity_Thm_Signed:ratio b} implies
    \begin{equation*}
        \Psi_{\mu}(\tau) \sim \Gamma(\rho + 1) l(1/\tau) \tau^{-\rho}  \quad\text{ as }\tau \to 0.  
    \end{equation*}
    Equations \eqref{eq:signedPsiLimit} and \eqref{eq:signedTheRelationship} follow immediately.
\end{proof}

We apply Theorem \ref{thm:Karamata_Tauberian_Signed} in the following example.

\begin{example}
Define $\mu \in \cM_\Psi$ via the density $f_{\mu}(x):= x\left(\frac{1}{2}+ \cos(x)\right)$. Then one can readily check that
\begin{equation*}
\Psi_{\mu}(\tau) = \frac{3\tau^4 + 1}{ 2(\tau^3 + \tau )^2 } =  \frac{1}{2\tau^2} + o(\tau^{-1}),\quad \tau \downarrow 0.
\end{equation*}
Thus, $\Psi_{\mu}$ is regularly varying with exponent $\rho = 2$. Noting that 
\begin{equation*}
    \liminf_{\tau \downarrow 0} \frac{\modd{\Psi_\mu(\tau)}}{\Psi_{\modd{\mu}}(\tau)} \geq \liminf_{\tau \downarrow 0} \frac{\frac{1}{2\tau^2}(1 + \tau o(\tau^{-1}))}{\frac{3}{2\tau^2}} = \frac{1}{3} > 0,
\end{equation*}
and
\begin{equation*}
    \limsup_{h \downarrow 0}\limsup_{\tau \downarrow 0}\modd{\frac{F_{\mu}(\tau^{-1}(x+h)) - F_{\mu}(\tau^{-1}x)}{\Psi_{\mu}(\tau)}} = \limsup_{h \downarrow 0}\limsup_{\tau \downarrow 0}\modd{\frac{ xh  +  h^2/2 + \tau^2o(\tau^{-1})}{1+ \tau^2o(\tau^{-1})}} = 0,
\end{equation*}
we see that both \eqref{eq:condition_1} and \eqref{eq:extraCondtion} are satisfied. Hence, by Theorem \ref{thm:Karamata_Tauberian_Signed} it follows that 
\begin{equation*}
    \Psi_{\mu}(\mu)(\tau) \sim \Gamma(2 + 1) F_\mu(\tau^{-1}),\quad \tau \downarrow 0.
\end{equation*}
\end{example}

\appendix
\section{Characterisation via Laplace transforms}\label{section:appendix}

In this appendix, we recall some useful facts from functional analysis and use them to prove that the measures in $\cM_{\Psi}$ are uniquely characterised by their Laplace transforms.

First, we state the Stone-Weierstrass Theorem in the version of de Branges \cite{de_branges_stone-weierstrass_1959}. To this end, recall that $\cC \subset C_0({\R_{+}})$ \emph{vanishes nowhere} on ${\R_{+}}$ if for all $x\in {\R_{+}}$, there exists $f \in \cC$ such that $f(x)\neq 0$. Moreover, $\cC$ \emph{separates points} if for each $x,y \in {\R_{+}}$ such that $x \neq y$, there exists $f \in \cC$ such that $f(x) \neq f(y)$.

\begin{theorem}[Stone-Weierstrass Theorem] \label{thm:stoneWeierstrass}
    Let $\cC$ be a subalgebra of $C_0({\R_{+}})$. Then $\cC$ is dense in $C_0({\R_{+}})$ (for the topology of uniform convergence) if and only if it separates points and vanishes nowhere.
\end{theorem}

Next, we recall the concept of a \emph{separating family} of measurable maps for the space $\cM$.

\begin{definition}
A family $\cC$ of measurable maps $f : {\R_{+}} \to \R$ is called a \emph{separating family} for $\cM$ if, for any two measures $\mu, \nu \in \cM$, the equality $\int f \de \mu = \int f \de \nu$ for all $ f\in\cC\cap L^1(\mu)\cap L^1(\nu)$ implies that $\mu = \nu$.
\end{definition}

Denote by $\textnormal{Lip}_1({\R_{+}},[0,1])$ the space of all functions $f: \R_+ \to [0, 1]$ that are Lipschitz with constant $1$. The following result is an immediate consequence of \cite[Theorem 13.11]{klenke_probability_2014}.
\begin{theorem}\label{thm:lipschitzSeparation}
$C_c({\R_{+}}) \cap \textnormal{Lip}_1({\R_{+}},[0,1])$ is a separating family for $\cM$.
\end{theorem}

Theorem \eqref{thm:lipschitzSeparation} together with Theorem \ref{thm:stoneWeierstrass} and the observation that a dense subset (for the topology of uniform convergence) of a separating family is again a separating family yields the following corollary.

\begin{corollary}\label{corollary:separatingClass}
Let $\cC$ be a subalgebra of $C_0({\R_{+}})$ that separates points and vanishes nowhere. Then $\cC$ is a separating family for $\cM$.
\end{corollary}

With the help of Corollary \eqref{corollary:separatingClass}, we can now state and prove the characterisation of $\cM_{\Psi}$ via Laplace transforms.

\begin{proposition}[Characterisation via Laplace Transforms]\label{thm:signedCharacterisation}
    Any $\mu \in \cM_{\Psi}$ is uniquely determined by its Laplace transform $\Psi_\mu$.
\end{proposition}
\begin{proof}
First assume that $\mu \in \cM$. Let $\cC' := \{e^{-\la x}:\la > 0\}$ and $\cC$ be the set of finite linear combinations of elements in $\cC'$. Then $\cC$ is a sub-algebra of $C_0({\R_{+}})$ that separates points and vanishes nowhere. Thus, $\cC$ is a separating family for $\cM$ by \cref{corollary:separatingClass}. Since the elements of $\cC'$ correspond to $\Psi_\mu$ for different values of $\lambda > 0$, the result follows.

Next, let $\mu \in \cM_{\Psi}$. Fix $\e >0$ and define $\mu^{(\e)} \in \cM(\R_{+})$ via $\mu^{(\e)}(\den x) := e^{-\e x} \mu(\den x) $. Then for $\la > 0$,
    \begin{equation}\label{eq:lapChar}
        \Psi_{\mu}(\la +\e) = \Psi_{\mu^{(\e)}}(\la).
    \end{equation}
Since $\mu^{(\e)} \in \cM$, it is uniquely determined by $\Psi_{\mu^{(\e)}}$. The latter is in turn uniquely determined by $\Psi_\mu$ by \eqref{eq:lapChar}. Since $\mu$ is  uniquely determined by  $\mu^{(\e)}$ and $\e$, the result follows.  
\end{proof}



\bibliographystyle{plain} 
\bibliography{Tauberian}

\end{document}